\documentclass[12pt]{amsart}
\usepackage{SSdefn}
\usepackage{eucal}
\usepackage[shortlabels]{enumitem}

\DeclareMathOperator{\Ext}{Ext}
\newcommand{\FIR}{\mathbf{FIR}}
\newcommand{\FI}{\mathbf{FI}}
\newcommand{\fgen}{\mathrm{fg}}

\newcommand{\stacks}[1]{\cite[\href{http://stacks.math.columbia.edu/tag/#1}{Tag~#1}]{stacks}}

\title{The semi-linear representation theory of the infinite symmetric group}

\author{Rohit Nagpal}
\address{Department of Mathematics, University of Michigan, Ann Arbor, MI}
\email{\href{mailto:rohitna@umich.edu}{rohitna@umich.edu}}
\urladdr{\url{http://www-personal.umich.edu/~rohitna/}}

\author{Andrew Snowden}
\address{Department of Mathematics, University of Michigan, Ann Arbor, MI}
\email{\href{mailto:asnowden@umich.edu}{asnowden@umich.edu}}
\urladdr{\url{http://www-personal.umich.edu/~asnowden/}}

\thanks{AS was supported by NSF DMS-1453893.}

\date{\today}

\begin{document}

\begin{abstract}
We study the category $\cA$ of smooth semilinear representations of the infinite symmetric group over the field of rational functions in infinitely many variables. We establish a number of results about the structure of $\cA$, e.g., classification of injective objects, finiteness of injective dimension, computation of the Grothendieck group, and so on. We also prove that $\cA$ is (essentially) equivalent to a simpler linear algebraic category $\cB$, which makes many properties of $\cA$ transparent.
\end{abstract}

\maketitle

\tableofcontents

\section{Introduction}

Cohen \cite{cohen,cohen2} proved that ideals in the infinite variable polynomial ring $R=\bk[\xi_1,\xi_2,\ldots]$ that are stable under the infinite symmetric group $\fS$ satisfy the ascending chain condition---that is, $R$ is $\fS$-noetherian---and applied this theorem to establish a finiteness property of the variety of metabelian groups. Cohen's theorem has received intense interest in the last decade (see, e.g., \cite{AschenbrennerHillar, DraismaEggermont, DraismaKuttler, GunturkunNagel, HillarSullivant, KLS, LNNR, LNNR2, NagelRomer, NagelRomer2}) due to its applicability in a wide variety of subjects. See \cite{draisma-notes} for a good introduction.

Motivated by this interest in $\fS$-ideals of $R$, and, more generally, $\fS$-equivariant $R$-modules, we have conducted a detailed investigation of these objects. We will report our results in a series of papers. This paper, which is the first, examines the structure of smooth semilinear representations of $\fS$ over the fraction field of $R$. Some of our results had previously been obtained by Rovinsky; see \S \ref{ss:rovinsky} for details.

\subsection{The structure of semilinear representations} \label{ss:results}

Fix a field $\bk$ and let $\bK=\bk(\xi_1,\xi_2,\ldots)$ be the rational function field in infinitely many variables. The infinite symmetric group $\fS=\bigcup_{n \ge 1} \fS_n$ naturally acts on $\bK$. We let $\cA$ be the category of smooth semilinear representations of $\fS$ over $\bK$ (see \S \ref{s:basic} for the definitions). Our goal is to understand this category.

Let $\bV^n$ be the $\bk$-vector space with basis indexed by $n$-element subsets of $[\infty]=\{1,2,\ldots\}$. This is naturally a smooth $\bk$-linear representation of $\fS$. Let $\bI^n=\bK \otimes_{\bk} \bV^n$, which is an object of $\cA$. We call an object of $\cA$ {\bf standard} if it is isomorphic to a direct sum of $\bI^n$'s. The standard objects play a central role in our investigation.

The following summarizes our main results about the structure of $\cA$. In what follows, $M$ denotes an arbitrary finitely generated object of $\cA$.
\begin{enumerate}[(i)]
\item The category $\cA$ is locally noetherian (Proposition~\ref{prop:noeth}). This is really just a corollary of Cohen's aforementioned work.
\item A sufficiently high shift of $M$ is standard (Theorem~\ref{thm:shift} and Corollary~\ref{cor:sstd}). The shift functor is defined in \S \ref{s:shift}.
\item $M$ can be embedded into a finitely generated standard object (Theorem~\ref{thm:embed}).
\item The standard objects are precisely the injective objects of $\cA$, and the $\bI^n$'s are precisely the indecomposable injectives (Theorem~\ref{thm:standards-are-injective} and Corollary~\ref{cor:indecomposable}).
\item $M$ has finite injective dimension, and admits a finite length resolution by finitely generated standard objects (Proposition~\ref{prop:res2}).
\item The Grothendieck group of $\cA^{\rm fg}$ (the category of finitely generated objects) has for a basis the classes $[\bI^n]$ for $n \ge 0$ (Theorem~\ref{thm:groeth}).
\end{enumerate}

\subsection{The equivalence with $\FIR^{\op}$-modules}

The above results show that the standard objects exert a great deal of control over the category $\cA$. This suggests it might be possible to completely describe $\cA$ in terms of the combinatorics of standard objects. We show that this is indeed the case, as we now explain.

Let $\FIR$ be the following $\bk$-linear category. The objects are finite sets. The morphism space $\Hom_{\FIR}(S,T)$ consists of all formal sums $\sum_{\phi \colon S \to T} c(\phi) [\phi]$ where $\phi$ varies over injections, $[\phi]$ is a formal symbol, and $c(\phi)$ belongs to the rational function field $\bk(t_i)_{i \in T}$. Composition is defined in the obvious manner (see \S \ref{s:fir}). The notation $\FIR$ is intended to be thought of as $\FI+\bR$, where $\FI$ is the category of finite sets and injections, as appearing in the work of Church--Ellenberg--Farb \cite{fimodules}, and $\bR$ indicates the presence of rational function coefficients. An {\bf $\FIR^{\op}$-module} is a $\bk$-linear functor from $\FIR^{\op}$ to the category of $\bk$-vector spaces. We let $\cB$ be the category of $\FIR^{\op}$-modules. The following is our main result on the connection between $\cA$ and $\cB$ (see Theorem~\ref{thm:main}):

\begin{theorem} \label{thm:equiv}
The category $\cA^{\rm fg}$ of finitely generated objects in $\cA$ is contravariantly equivalent to the category $\cB^{\rm fp}$ of finitely presented objects in $\cB$.
\end{theorem}

This is a useful description of semilinear representations since $\FIR^{\op}$-modules tend to be easy to understand. For example, using it one can immediately write down the entire lattice of subobjects of $\bI^1$ (Example~\ref{ex:I1}). One can also easily recover many of the structural results about $\cA$ through this equivalence (see \S \ref{ss:app} for an example), though many of those results are needed to establish the equivalence in the first place.

Theorem~\ref{thm:equiv} implies that finitely presented objects of $\cB$ are coherent and artinian, and we establish these properties before proving the equivalence. We note that $\cB$ is \emph{not} locally noetherian. We believe that $\FIR^{\op}$ is the first natural category appearing in the context of representation stability where the module category is locally coherent but not locally noetherian.

\subsection{Relation to the work of Rovinsky} \label{ss:rovinsky}

After completing this work, we learned of the recent work of Rovinsky \cite{rovinsky,rovinsky2,rovinsky3} on the semilinear representation theory of $\fS$. Some of our results appear in Rovinsky's work: in particular, he proves (i), (ii), (iii) and (iv) from \S \ref{ss:results}. Results (v) and (vi), as well as the connection to $\FIR^{\op}$-modules, appear to be new. Our proofs of (i)--(iv) are quite different from Rovinsky's. For example, we prove vanishing of $\Ext^1$ between standard objects through a cohomology calculation (see \S \ref{ss:coh}), while Rovinsky's argument relies on growth estimates of certain dimensions \cite[\S 3.1]{rovinsky2}.

One further difference between Rovinsky's work and ours is that he employs a more general setup. Let $F/\bk$ be an extension in which $\bk$ is algebraically closed, and let $F_{\infty}$ be the fraction field of the infinite tensor power of $F$ over $\bk$. For example, if $F=\bk(\xi)$ then $F_{\infty}=\bk(\xi_1,\xi_2,\ldots)$. Rovinsky studies semilinear representations of $\fS$ over $F_{\infty}$, for quite general $F$. This extra level of generality does not appear to affect the proofs or results greatly, however, and it seems that our arguments would apply in that setting as well. We have not attempted to carry this out though.

\subsection{Notation}
\label{ss:note}

We list some of the important notation here:
\begin{description}[align=right,labelwidth=2cm,leftmargin=!]
\item [$\fS_n$] the symmetric group on $n$ letters
\item [$\fS$] the infinite symmetric group $\bigcup_{n \ge 1} \fS_n$
\item [$\fS^n$] the subgroup of $\fS$ fixing each of $1, \ldots, n$
\item [$\bk$] the base field
\item [$\bK$] the field $\bk(\xi_1,\xi_2,\ldots)$ of rational functions
\item [$\cA$] the category of smooth semilinear representations of $\fS$ over $\bK$
\item [$\bV^r$] the $\bk$-vector space with basis indexed by $r$-element subsets of $[\infty]$
\item [$\bW^r$] the $\bk$-vector space with basis $e_{i_1,\ldots,i_r}$ where $i_1,\ldots,i_r$ are distinct elements of $[\infty]$
\item [$\bI^r$] the $\bK$-vector space $\bK \otimes_{\bk} \bV^r$
\item [$\bJ^r$] the $\bK$-vector space $\bK \otimes_{\bk} \bW^r$
\item [{$[n]$}] the set $\{1,\ldots,n\}$, also used for $n=\infty$
\item [$\cB$] the category of $\FIR^{\op}$-modules
\end{description}

\subsection{Outline}

In \S \ref{s:basic} we give the basic definitions and prove some fundamental results. In \S \ref{s:shift} we introduce the shift operator and prove the important shift theorem. In \S \ref{s:std} we prove our main structural results on semilinear representations. Finally, in \S \ref{s:fir} we establish the connection to $\FIR^{\op}$-modules.

\section{Basic definitions and results} \label{s:basic}

\subsection{Semilinear representations}

Let $K$ be a field on which a group $G$ acts by field homomorphisms. A {\bf $K$-semilinear representation} of $G$ is a $K$-vector space $V$ equipped with an additive action of $G$ such that $g(av)=(ga)(gv)$ holds for all $g \in G$, $a \in K$, and $v \in V$. Galois descent can be used to describe semilinear representations in some cases:

\begin{proposition} \label{prop:galois}
Suppose $G$ is finite and acts faithfully on $K$, and let $F=K^G$ be the fixed field. Then we have mutually quasi-inverse equivalences of categories
\begin{align*}
\{ \text{$K$-semilinear representations of $G$} \} &\cong \{ \text{$F$-vector spaces} \} \\
M &\mapsto M^G \\
K \otimes_F V &\mapsfrom V
\end{align*}
In particular, every $K$-semilinear representation of $G$ is isomorphic to a direct sum of copies of the trivial semilinear representation $K$.
\end{proposition}

\begin{proof}
See \stacks{0CDQ}.
\end{proof}

\subsection{The category $\cA$}

We let $\fS_n$ be the symmetric group on $n$ letters, and let $\fS=\bigcup_{n \ge 1} \fS_n$ be the infinite symmetric group. We let $\fS^n$ be the subgroup of $\fS$ stabilizing each of the numbers $1, 2, \ldots, n$. We say that an action of $\fS$ on a set $X$ is {\bf smooth} if every element $x \in X$ is stabilized by $\fS^n$ for some $n$ (depending on $x$).

Let $\bk$ be a field. We let $\bK=\bK(\bk)$ be the rational function field $\bk(\xi_1, \xi_2, \ldots)$, on which the group $\fS$ naturally acts smoothly. Let $\cA=\cA_{\bk}$ denote the category of smooth semilinear representations of $\fS$ over $\bK$. This is the main object of study of this paper. Note that there is nothing like Proposition~\ref{prop:galois} that applies in this case, since the group is infinite.

\subsection{Standard objects}

Let $\bV^r=\bV^r_{\bk}$ be the $\bk$-vector space with basis $\{e_S\}$, where $S$ varies over $r$-element subsets of $[\infty]=\{1,2,\ldots\}$. This is naturally a smooth $\bk$-linear representation of $\fS$. We let $\bI^r=\bK \otimes_{\bk} \bV^r$, which is naturally a smooth $\bK$-semilinear representation of $\fS$, and thus an object of $\cA$. We let $\epsilon_r$ be the basis vector $e_{\{1,\ldots,r\}}$, which clearly generates both $\bV^r$ and $\bI^r$. We say that an object of $\cA$ is {\bf standard} if it is a direct sum of objects of the form $\bI^r$, and {\bf semi-standard} if it admits a finite filtration whose graded pieces are standard.

\begin{proposition} 
\label{prop:adjunction}
Let $V$ be a smooth $\bk$-linear  $\fS$-representation, and let $M$ be an object in $\cA$. Then we have a natural isomorphism
\begin{displaymath}
\phi \colon \Hom_{\fS}(V, M) \to \Hom_{\cA}(\bK \otimes_{\bk} V, M)
\end{displaymath}
given by $\phi(g)(x \otimes v) = xg(v)$. In particular, an injective object in $\cA$ is injective as a smooth $\fS$-representation.
\end{proposition}

\begin{proof}
Define
\begin{displaymath}
\psi \colon \Hom_{\cA}(\bK \otimes_{\bk} V, M) \to \Hom_{\fS}(V, M)
\end{displaymath}
by $\psi(f)(v) = f(1 \otimes v)$. One easily sees that $\psi$ is the inverse of $\phi$.
\end{proof}

\begin{proposition} \label{prop:stdmap}
Let $M$ be an object of $\cA$. Then the map
\begin{displaymath}
\Hom_{\cA}(\bI^r, M) \to M^{\fS_r \times \fS^r}, \qquad
f \mapsto f(\epsilon_r)
\end{displaymath}
is an isomorphism.
\end{proposition}

\begin{proof}
By Proposition~\ref{prop:adjunction}, we have
\begin{displaymath}
\Hom_{\cA}(\bI^r, M) = \Hom_{\fS}(\bV^r, M).
\end{displaymath}
Now, note that $\bV^r$ is the induction of the trivial representation of $\fS_r \times \fS^r$ to $\fS$. The result thus follows from Frobenius reciprocity.
\end{proof}


\begin{corollary} \label{cor:stdquo}
Every object of $\cA$ is a quotient of a standard object.
\end{corollary}

\begin{proof}
Let $M$ be an object of $\cA$ and let $x \in M$. Since $M$ is smooth, $x$ is fixed by $\fS^r$ for some $r$. The space $M^{\fS^r}$ is a semilinear representation of $\fS_r$ over the field $\bK_r=\bk(\xi_1, \ldots, \xi_r)$, and is therefore spanned as a $\bK_r$-vector space by its $\fS_r$-invariants (Proposition~\ref{prop:galois}). We can thus write $x=\sum_{i=1}^n a_i y_i$ where $a_i \in \bK_r$ and $y_i \in M^{\fS_r \times \fS^r}$. We thus see that if $f_i \colon \bI^r \to M$ is the map corresponding to $y_i$ then $x$ belongs to $\sum_{i=1}^n \im(f_i)$. Since $x$ was arbitrary, it follows that the natural map
\begin{displaymath}
\bigoplus_{r \ge 0} \Hom_{\cA}(\bI^r, M) \otimes \bI^r \to M
\end{displaymath}
is surjective. Since the domain is standard, the result follows.
\end{proof}

We define the {\bf generation degree} of $M$, denoted $g(M)$, to be the minimum $n$ such that $M$ is a quotient of a sum of $\bI_{\bk}^r$'s with $r \le n$, or $\infty$ if no such $n$ exists. We note that $M$ has generation degree $\le g$ if and only if $M$ is generated by $M^{\fS^g}$.

\subsection{The noetherian property}

\begin{lemma}
\label{lem:cohen}
Let $\bR=\bk[\xi_1,\xi_2,\ldots]$, and let $V$ be the $\bk$-linear representation of $\fS$ with basis $\{e_S\}$ where $S$ varies over $r$-element subsets of $[\infty]$. Then $\bR \otimes_{\bk} V$ is a noetherian object in the category of $\fS$-equivariant $\bR$-modules. Moreover, the field $\bk$ can be replaced by any noetherian ring.
\end{lemma}

\begin{proof}
The noetherian property for ideals is well-known; see \cite{cohen2, AschenbrennerHillar}. The result for modules can be deduced by the same arguments. Alternatively, the module result can be deduced from the ideal result, as we now explain.

Let $\bR' = \bR[\zeta_1, \zeta_2, \ldots]$, on which $\fS$ acts by permuting both sets of variables. Let $I_k$ be the ideal of $\bR'$ generated by monomials in $\zeta_1, \zeta_2, \ldots$ of degree $k$ and the monomials of the form $\zeta_i^2$.  It is known that $\fS$-stable ideals of $\bR'$ satisfy ACC \cite{cohen2, AschenbrennerHillar}. In particular, the $\bR'$-module $I_r/I_{r+1}$ is equivariantly noetherian. Now we note that the map
\begin{displaymath}
\{ \text{subobjects of $I_r/I_{r+1}$} \} \to \{ \text{subobjects of $\bR \otimes_{\bk} V$} \},
\end{displaymath}
given by the pullback along the natural $\fS$-equivariant injection $i \colon \bR \to \bR'$, is an order preserving bijection. Since the source satisfies ACC, so does the target. As this argument and the cited work apply when $\bk$ is an arbitrary noetherian ring, the proof is complete.
\end{proof}

\begin{proposition} \label{prop:noeth}
The category $\cA$ is locally noetherian.
\end{proposition}

\begin{proof}
By Corollary~\ref{cor:stdquo}, it suffices to show that $\bI^r$ is noetherian for each $r$. Let $V$ be as in the lemma above. Thus $\bI^r=\bK \otimes_{\bk} V$. Let $\bR=\bk[\xi_1,\xi_2,\ldots]$, so that $\bK=\Frac(\bR)$. The map
\begin{displaymath}
\{ \text{subobjects of $\bK \otimes_{\bk} V$} \} \to \{ \text{subobjects of $\bR \otimes_{\bk} V$} \}, \qquad
M \mapsto M \cap (\bR \otimes_{\bk} V)
\end{displaymath}
is injective and order-preserving. Since the target satisfies ACC (by the previous lemma), so does the source.
\end{proof}

\section{The shift theorem} \label{s:shift}

\subsection{Definitions}

We now define a functor
\begin{displaymath}
\Sigma^n \colon \cA_{\bk} \to \cA_{\bk(u_1, \ldots, u_n)},
\end{displaymath}
called the $n$th {\bf shift functor}. Let $M \in \cA_{\bk}$. The additive group underlying $\Sigma^n(M)$ is simply $M$. For notational clarity, we write $x^{\flat(n)}$ for the element $x \in M$ when we regard it as an element of $\Sigma^n(M)$. The $\bK(u_1,\ldots,u_n)$-linear structure on $M^{\flat}$ is defined as follows:
\begin{displaymath}
u_i \cdot x^{\flat(n)} = (\xi_i x)^{\flat(n)}, \qquad
\xi_i \cdot x^{\flat(n)} = (\xi_{i+n} x)^{\flat(n)}.
\end{displaymath}
For $\sigma \in \fS$, let $\sigma^{\sharp(n)} \in \fS$ be the element define by
\begin{displaymath}
\sigma^{\sharp(n)}(i) = \begin{cases}
1 & \text{if $1 \le i \le n$} \\
\sigma(i-n)+n & \text{if $i>n$} \end{cases}
\end{displaymath}
Note that $\sigma \mapsto \sigma^{\sharp(n)}$ is an isomorphism $\fS \to \fS^n$. The $\fS$-action on $\Sigma^n (M)$ is defined by
\begin{displaymath}
\sigma \cdot x^{\flat(n)}=(\sigma^{\sharp(n)} \cdot x)^{\flat(n)}.
\end{displaymath}
It is clear that the shift functor is exact and cocontinuous. We note that $\Sigma^1  (\bI^r_{\bk})$ is naturally isomorphic to $ \bI^r_{\bk(u_1)} \oplus \bI^{r-1}_{\bk(u_1)}$, and that we have a natural isomorphism $\Sigma^n \circ \Sigma^m \cong \Sigma^{n+m}$.  It immediately follows that the shift functor takes standard objects to standard objects, preserves finite generation, and does not increase generation degree.  

Let $\Omega^n \colon \cA_{\bk} \to \cA_{\bk(u_1,\ldots,u_n)}$ be the functor given by extension of scalars, that is, $\Omega(M)=\bK(u_1,\ldots,u_n) \otimes_{\bK} M$. We have a natural isomorphism $\Omega^n \circ \Omega^m \cong \Omega^{n+m}$.

We have a natural transformation $\Omega^n \to \Sigma^n$ that on a representation $M$ is the map
\begin{displaymath}
\phi_M \colon \bK(u_1,\ldots,u_n) \otimes_{\bK} M \to \Sigma^n(M), \qquad
a \otimes x \mapsto a \cdot (\tau^n x)^{\flat(n)},
\end{displaymath}
where $\tau(i)=i+1$. Note that $\tau$ is not an element of $\fS$, but $\tau$ naturally acts on any smooth $\fS$-set: for an element $x$, we define $\tau x$ to be $(1\;2\;3\;\cdots\;m) x$ for $m \gg 0$, the result being independent of $m$ for $m$ large. We verify that $\phi_M$ is indeed a morphism in $\cA_{\bk(u_1,\ldots,u_n)}$. It is clearly $\bk(u_1,\ldots,u_n)$-linear. We have 
\begin{displaymath}
\phi_M(\xi_i (a \otimes x))= \phi_M(a \otimes (\xi_i x)) = a\cdot (\tau^n \xi_i x)^{\flat(n)}=a \cdot (\xi_{i+n} \tau^n x)^{\flat(n)} = a\xi_i \cdot (\tau^n x)^{\flat(n)} = \xi_i \cdot \phi_M(a \otimes x).
\end{displaymath}
For $\sigma \in \fS$ and $x \in M$, we have $\tau^n \sigma x = \sigma^{\sharp(n)} \tau^n x$, and so 
\begin{align*}
\phi_M(\sigma (a \otimes x)) = \phi_M(\sigma(a) \otimes \sigma (x)) &= \sigma(a)\cdot(\tau^n \sigma x)^{\flat(n)} = \sigma(a) \cdot (\sigma^{\sharp(n)} \tau^n x)^{\flat(n)}\\
&=\sigma(a) \cdot (\sigma \cdot((\tau^n x)^{\flat(n)}))  =\sigma \cdot (a \cdot (\tau^n x)^{\flat(n)}) = \sigma \cdot \phi_M(a \otimes x).
\end{align*}
%
Thus $\phi_M$ is indeed a morphism in $\cA_{\bk(u_1,\ldots,u_n)}$.

\begin{proposition}
	\label{prop:injectivity-phi}
The morphism $\phi_M$ is injective. If $M$ is finitely generated then the cokernel of $\phi_M$ has strictly smaller generation degree than $M$.
\end{proposition}

We need a lemma:

\begin{lemma}
	Let $M$ be an object of $\cA_{\bk}$. Then $\bk(\xi_1, \ldots, \xi_m)$-linearly independent elements of $M^{\fS^m}$ are $\bK$-linearly independent. 
\end{lemma}

\begin{proof}
	Let $x_1, \ldots, x_r \in M^{\fS^m}$ be $\bk(\xi_1, \ldots, \xi_m)$-linearly independent. Suppose we have a linear dependence
	\begin{displaymath}
	\sum_{i=1}^s a_i x_i = 0
	\end{displaymath}
	with $a_i \in \bK$, and choose one with $s$ minimal. Dividing by $a_s$, we can assume $a_s=1$. Now, there is some $i$ for which $a_i \not\in \bk(\xi_1, \ldots, \xi_m)$. We can therefore find $\sigma \in \fS^n$ such that $\sigma a_i \ne a_i$. Applying $\sigma-1$ to our relation thus gives a smaller non-zero relation, a contradiction.
\end{proof}

\begin{proof}[Proof of Proposition~\ref{prop:injectivity-phi}] We first show that $\phi_M$ is injective. Suppose that $\sum_{i=1}^r a_i \otimes x_i$ is a nonzero element that belongs to the kernel of $\phi_M$. Let $m \gg 0$ be sufficiently large so that the $x_i$ and $a_i$ are all $\fS^m$-invariant. Without loss of generality, we may assume that $x_i$ are $\bk(\xi_1, \ldots, \xi_m)$-linearly independent. Letting $\sigma=(1\;2\;\cdots\;m+n)$, we have $\tau^n x_i=\sigma^n x_i$ for all $i$ (where $\tau(i) = i+1$, as defined above). Thus
	\begin{displaymath}
	0=\phi_M \left( \sum_{i=1}^r a_i \otimes x_i \right) = \sum_{i=1}^r a_i \cdot (\sigma^n x_i)^{\flat(n)} = \sum_{i=1}^r a_i \cdot (\sigma^n x_i)^{\flat(n)}.
	\end{displaymath} Let $b_i$ be the rational funtion obtained by replacing $u_j$ with $\xi_{m + j}$ for $1 \le j \le n$. Applying $\sigma^{-n}$ to the equation above, we find
	\begin{displaymath}
	\sum_{i=1}^r b_i x_i = 0. 
	\end{displaymath} This contradicts the lemma above. Thus the kernel of $\phi_M$ is trivial, completing the proof.

We now show that the cokernel of $\phi_M$ has strictly smaller generation degree than $M$. For this, note that the functor taking $M$ to the cokernel of $\phi_M$ is exact and cocontinuous. It therefore suffices to prove the assertion for $M=\bI^r_{\bk}$, which is a simple computation.
\end{proof}

\subsection{The shift theorem}

We say that $M \in \cA_{\bk}$ has {\bf stable generation degree $\le n$} if some iterated shift of $M$ has generation degree $\le n$. It is clear that $\bI^r_{\bk}$ has stable generation degree $r$. The following is the main result on this invariant:

\begin{proposition}
Any proper quotient of $\bI^r_{\bk}$ has stable generation degree $<r$.
\end{proposition}

\begin{proof}
Consider a quotient $\bI^r_{\bk}/K$ with $K$ a non-zero submodule. It is not hard to see that, for some $n$, $K$ contains some element that is a linear combination of basis vectors corresponding to subsets of $[n+r]$ with the basis vector corresponding to $\{n+1,\ldots,n+r\}$ having coefficient~1. The $n$-fold shift of $\bI^r_{\bk}$ decomposes as $\bI^r_{\bk(u_1,\ldots,u_n)} \oplus C$ where $C$ is a standard object generated in degree less than $r$; moreover, the copy of $\bI^r_{\bk(u_1,\ldots,u_n)}$ is generated by the basis vector corresponding to $\{n+1,\ldots,n+r\}$. We thus see that the $n$-fold shift of $\bI^r_{\bk}/K$ is a quotient of $C$, which proves the theorem.
\end{proof}

\begin{theorem} \label{thm:shift}
If $M$ is a finitely generated object of $\cA_{\bk}$ then $\Sigma^n(M)$ is semi-standard for some $n$.
\end{theorem}

\begin{proof}
We proceed by induction on stable generation degree. Thus suppose the theorem has been proven in stable generation degree $<g$, and let us prove it in degree $g$. Let $M$ be given of stable generation degree $g$. We may as well replace $M$ by a shift, and assume that $M$ is generated in degree $g$. In fact, we may assume that $M$ is a quotient of $\bI^g_{\bk}$, as $M$ is admits a finite filtration with graded pieces of this form. Now, if $M=\bI^g_{\bk}$ then there is nothing to prove. Otherwise, $M$ has strictly smaller generation degree by the previous theorem, and we are done by the inductive hypothesis.
\end{proof}

\section{The structure of semilinear representations} \label{s:std}

\subsection{A cohomology calculation} \label{ss:coh}

\begin{lemma}
Let $f$ be a rational function in two variables over $\bk$ satisfying \[f(x,y) + f(y,z) = f(x, z).\] Then $f(x,z) = g(x) + h(z)$ for some one variable rational functions $g, h$ over $\bk$.
\end{lemma}
\begin{proof}
Since $\ol{\bk}$ is infinite, we can pick a value of $y$ in $\ol{\bk}$ such that denominators in $f(x,y)$ and $f(y,z)$ do not vanish. Plugging this value in the given relation, we see that $f(x,z) = g(x) + h(z)$ for some one variable rational functions $g, h$ over $\ol{\bk}$. Suppose $\sigma \in \Aut(\ol{\bk} / \bk) $, then we have $\sigma(g(x)) + \sigma(h(z)) = g(x) + h(z)$. Since $\ol{\bk}(x) \cap \ol{\bk}(z) = \ol{\bk}$, we conclude that $\sigma g(x) = g(x)$ and $\sigma h(z) = h(z)$. So $g(x)$ and $h(z)$ are defined over $\bk$, completing the proof.
\end{proof}

\begin{proposition}
Let $c \colon \fS \to \bI^r$ be a 1-cocycle for which there exists $n$ such that $c(\sigma)=0$ for all $\sigma \in \fS^n$. Then $c$ is a coboundary.
\end{proposition}

\begin{proof}
First suppose that $n=1$, i.e., $c(\sigma)=0$ for $\sigma \in \fS^1$. Suppose $m \ge 1$. Assume that $\sigma \in \fS_m$ and $\tau \in \fS^m \subset \fS^1$. Then we have \[ \tau c(\sigma) = c(\tau \sigma) = c(\sigma \tau) = c(\sigma).   \] This shows that $c(\sigma)$ is an $\fS^m$ invariant. In particular, $c((1\;2))$ is invariant under $\fS^2$. We also have the following equations: \begin{align*}
&0= c(1) = c((1\;2)(1\;2)) = c((1\;2)) + (1\; 2)c((1\;2)) \\
&c((1\;2)) = c((1\;2)(2\;3)) = c((1\;3)(1\;2)) = c((1\;3)) + (1\;3) c((1\;2)).
\end{align*} We now proceed in cases:

\textit{Case 1: $r=0$.} Suppose $c((1\;2)) = f(\xi_1, \xi_2)$. From the equations above, we have $f(\xi_1, \xi_2) = -f(\xi_2, \xi_1)$, and $f(\xi_1, \xi_3) + f(\xi_3, \xi_2) = f(\xi_1, \xi_2)$. Thus, by the previous lemma, $f(\xi_1, \xi_2)$ is of the form $g(\xi_1) - g(\xi_2)$ for some $g(\xi_1) \in \bk(\xi_1)$. For $n \ge 2$, we have \[ c((1\;n)) = c((2\;n)(1\;2)(2\;n)) = (2 \; n) f(\xi_1, \xi_2) = f(\xi_1, \xi_n) = g(\xi_1) - (1\;n) g(\xi_1).  \]  But then $c$ agrees with the coboundary $g(\xi_1) - \sigma g(\xi_1)$ on all transpositions, and so $c$ is a coboundary. 

\textit{Case 2: $r=1$.} The $\fS^2$-invariants of $\bI^1$ consist of $\bk(\xi_1,\xi_2)$-linear combinations of $e_1$ and $e_2$. We can thus write $c((1\;2))=f(\xi_1,\xi_2) e_1 + g(\xi_1,\xi_2) e_2$. From the equations above, we see that
\begin{align*}
f(\xi_1,\xi_2) e_1 + g(\xi_1,\xi_2) e_2 &=  -g(\xi_2, \xi_1)e_1 - f(\xi_2, \xi_1)e_2 \\ 
f(\xi_1,\xi_2)e_1 + g(\xi_1,\xi_2) e_2&= f(\xi_1,\xi_3)e_1+g(\xi_1,\xi_3)e_3 + f(\xi_3,\xi_2) e_3+g(\xi_3,\xi_2) e_2.
\end{align*} We thus conclude that $f(\xi_1,\xi_2)=h(\xi_1)$, for some $h$, and $g(\xi_1,\xi_2)=-h(\xi_2)$. As in the previous case, we see that $c$ agrees with the coboundary  $h(\xi_1) e_1 - \sigma h(\xi_1) e_1$ on all transpositions, and so $c$ is a coboundary.

\textit{Case 3: $r=2$.} We have $c((1\;2))=f(\xi_1,\xi_2)e_{1,2}$. From the equations above, we see that
\begin{align*}
f(\xi_1,\xi_2)e_{1,2} &= - f(\xi_2, \xi_1)e_{1,2} \\
f(\xi_1,\xi_2)e_{1,2} &= f(\xi_1,\xi_3) e_{1,3} + f(\xi_3,\xi_2) e_{2,3}.
\end{align*}
Obviously, we must have $f=0$, and so $c=0$.

\textit{Case 4: $r>2$.} In this case, $\bI^r$ has no $\fS^2$-invariants, and so $c=0$.

We have thus completed the $n=1$ case. We now treat the general case. It suffices to show that $c \vert_{\fS^{n-1}}$ is a coboundary. Relabeling (note that the $(n-1)$-fold shift of $\bI^r_{\bk}$ is a standard object in $\cA_{\bk(u_1,\ldots, u_{n-1})}$), we are back to the $n=1$ case.
\end{proof}

\begin{remark}
The proposition can be interpreted as saying that a smooth version of the first cohomology of $\fS$ with coefficients in $\bI^r$ vanishes.
\end{remark}

\subsection{An Ext calculation}

\begin{proposition}
We have $\Ext^1(\bI^r, \bI^s)=0$ for all $r,s$.
\end{proposition}

\begin{proof}
Consider an extension
\begin{displaymath}
0 \to \bI^s \to E \stackrel{\pi}{\to} \bI^r \to 0.
\end{displaymath}
Let $e=e_{1,\ldots,r}$ be the generator of $\bI^r$, and let $f \in E$ be a lift of it. Since $e$ is invariant under $\fS^r$, we have $\sigma f = f + c(\sigma)$ for some $c(\sigma) \in \bI^s$, for $\sigma \in \fS^r$. Clearly, $c$ is a 1-cocycle. It is therefore trivial by the previous theorem. We can therefore choose $f$ so that it is $\fS^r$-invariant.

Now, $E^{\fS^r}$ is a semi-linear representation of the symmetric group $\fS_r$ on $r$ letters over the field $\bK_r$. By Galois descent, it is spanned over $\bK_r$ by its $\fS_r$-invariants. We can therefore write $f=\sum_{i=1}^n a_i f_i$ with $a_i \in \bK_r$ and $f_i \in E^{\fS_r \times \fS^r}$. Since $\pi(f)$ is non-zero, we must have $\pi(f_i) \ne 0$ for some $i$. Since $\pi(f_i)$ belongs to $(\bI^r)^{\fS_r \times \fS^r}$, it is a $\bF_r$-scalar multiple of $e$. Thus, replacing $f$ with some scalar multiple of $f_i$, we may as well assume that our lift $f$ of $e$ is $\fS_r$-invariant.

The mapping property of $\bI^r$ (Proposition~\ref{prop:stdmap}) now provides us with a map $\bI^r \to E$ taking $e$ to $f$. This map splits $\pi$.
\end{proof}

\begin{corollary} \label{cor:sstd}
Semi-standard objects are standard.
\end{corollary}

\subsection{The embedding theorem}

\begin{lemma} \label{lem:noeth}
Let $S$ be a localization of $\bk[u,\xi_1,\xi_2,\ldots]$ at some $\fS$-stable multiplicative set. Then the category of smooth $\fS$-equivariant $S$-modules is locally noetherian.
\end{lemma}

\begin{proof}
The proof is similar to the proof that $\cA$ is locally noetherian (Proposition~\ref{prop:noeth}). First suppose that $M$ is a smooth equivariant $S$-module. Regarding $M$ as an object of $\cA$, there is a standard object $I$ and a surjection $I \to M$. Thus the canonical $S$-linear map $S \otimes_{\bK} I \to M$ is also surjective. It follows that if $M$ is finitely generated then it is a quotient of a finite direct sum of objects of the form $S \otimes_{\bK} \bI^r$, and so it suffices to show that these are noetherian.

Let $V$ be the $\bk$-linear representation of $\fS$ used in Proposition~\ref{prop:noeth}. It thus suffices to show that $S \otimes_{\bk} V$ is noetherian. Let $R=\bk[u,\xi_1,\xi_2,\ldots]$. Then $R \otimes_{\bk} V$ is noetherian in the category of $\fS$-equivariant $R$-modules (Lemma~\ref{lem:cohen}). Since the map
\begin{displaymath}
\{ \text{subobjects of $S \otimes_{\bk} V$} \} \to \{ \text{subobjects of $R \otimes_{\bk} V$} \}, \qquad
M \mapsto M \cap (R \otimes_{\bk} V)
\end{displaymath}
is injective, the result follows.
\end{proof}

\begin{theorem} \label{thm:embed}
Let $M$ be a finitely generated object of $\cA$. Then there is an injection $f \colon M \to I$ where $I$ is a finitely generated standard object. Moreover, one can take $I$ to have the same generation as $M$, and, if $\bk$ is infinite, one can take $f$ such that $\coker(f)$ has strictly smaller generation degree.
\end{theorem}

\begin{proof}
We first assume $\bk$ is infinite. We say that $M$ is ``good'' if the theorem holds for $M$. By Theorem~\ref{thm:shift}, we know that $\Omega^n(M)$ is good for some $n$. It thus suffices to show that $M$ is good if $\Omega(M)$ is. We now do this.

Assume $\Omega(M)$ is good. We thus have an injection $f_0 \colon \Omega(M) \to \Omega(I)$ where $I$ is a standard object of $\cA$ of the same generation degree $g$ as $M$ and $\coker(f)$ has smaller generation degree; we note that $\Omega$ does not change generation degree, and that every standard object of $\cA_{\bk(u)}$ comes from applying $\Omega$ to a standard object of $\cA$. Let $x_1, \ldots, x_n$ be generators for $M$, and write $f_0(x_i) = \sum_j a_{i,j} \otimes b_{i,j}$ where $a_{i,j} \in \bK(u)$ and $b_{i,j} \in I$. Let $S$ be the localization of $\bK[u]$ obtained by inverting the $\fS$-orbits of the denominators of the $a_{i,j}$'s. Thus $f_0$ defines a map $f \colon S \otimes_{\bK} M \to S \otimes_{\bK} I$, which we denote by $f$.

Let $C$ be the cokernel of $f$. Let $y_1, \ldots, y_m$ be elements of $C$ that are invariant under $\fS^{g-1}$ that generate $\coker(f_0) = \bK(u) \otimes_S C$. Let $z_1, \ldots, z_{\ell}$ be generators for $C$. Then we can write $z_i = \sum c_{i,j} y_j$ for some $c_{i,j} \in \bK(u)$. Let $S'$ be the localization of $S$ obtained by inverting the $\fS$-orbits of the denominators of the $c_{i,j}$. Thus $S' \otimes_S C$ is generated by the $y$'s, and thus has generation degree $<g$. In what follows, we replace $S$ with $S'$.

Now, let $C' \subset C$ be the set of all elements with non-zero annihilator in $\bk[u]$. This is an $\fS$-stable $S$-submodule of $C$. Since $C$ is noetherian (Lemma~\ref{lem:noeth}), it follows that $C'$ is finitely generated. Thus there is a single non-zero polynomial in $\bk[u]$ that annihilates $C$. We replace $S$ with its localization at this polynomial. Thus $C$ is now $\bk[u]$-flat.

By consturction, $S$ is the localization of $\bK[u]$ obtain by inverting the $\fS$-orbits of finitely many elements, say $h_1, \ldots, h_t$, of $\bk[u,\xi_i]$. Regarding $h_i$ as a polynomial in the $\xi$'s with coefficients in $\bk[u]$, let $H_i \in \bk[u]$ be the greatest common divisor of its coefficients, and let $H$ be the least common multiple of the $H_i$'s. Since $\bk$ is infinite, we can find $a \in \bk$ such that $H(a) \ne 0$. Thus, upon setting $u=a$ none of the $h_i$'s (or any element of the $\fS$-orbit of $h_i$) becomes~0. It follows that $S/(u-a)S \cong \bK$.

We now take the short exact sequence
\begin{displaymath}
0 \to S \otimes_{\bK} M \to S \otimes_{\bK} I \to C \to 0
\end{displaymath}
and apply $- \otimes_{\bk[u]} \bk[u]/(u-a)$. Since $C$ is $\bk[u]$-flat, the sequence remains exact. We thus obtain a short exact sequence
\begin{displaymath}
0 \to M \to I \to C/(u-a)C \to 0.
\end{displaymath}
Since $C/(u-a) C$ is a quotient of $C$, it still has generation degree strictly less than that of $M$. This completes the proof when $\bk$ is infinite.

Suppose now that $\bk$ is finite. By the above proof, we can find an injection $\ol{\bk} \otimes_{\bk} M \to \ol{\bk} \otimes_{\bk} I$ for some finitely generated standard object $I$. This map is defined over some finite extension $\bk'$ of $\bk$. We thus have an injection $\bk' \otimes_{\bk} M \to \bk' \otimes_{\bk} I$. But $M$ injects into $\bk' \otimes_{\bk} M$, and if we regard $\bk' \otimes_{\bk} I$ as an object of $\cA$ it is simply a finite direct sum of copies of $I$. This completes the proof.
\end{proof}

\begin{corollary} \label{cor:res}
Suppose $\bk$ is infinite. Let $M \in \cA$ be finitely generated. Then we have a resolution
\begin{displaymath}
0 \to M \to I^0 \to \cdots \to I^n \to 0
\end{displaymath}
where each $I^k$ is standard and finitely generated. Moreover, if $M$ has generation degree $g$ then one can take $n=g$ and $I^k$ to have generation degree $\le g-k$.
\end{corollary}

\begin{remark}
In fact, the corollary remains true if $\bk$ is finite. See \S \ref{ss:finite} and \S \ref{ss:app}.
\end{remark}

\subsection{The injectivity theorem}

\begin{theorem}
\label{thm:standards-are-injective}
Standard objects are injective.
\end{theorem}

\begin{proof}
Let us show that $\bI^r$ is injective. Consider an injection $\bI^r \to M$, and let us show that it splits. It suffices to treat the case where $M$ is finitely generated, by a version of Baer's theorem (see, e.g., \cite[Proposition~A.4]{increp}). Choose an injection $M \to J$ with $J$ standard and finitely generated, which is possible by Theorem~\ref{thm:embed}. It suffices to split the injection $\bI^r \to J$, since restricting the splitting to $M$ will split the injection $\bI^r \to M$.

Decompose $J$ as $\bigoplus_{i=1}^n \bI^{a(i)}$. We can assume that our map $\bI^r \to J$ is non-zero on each factor, as we can simply discard the factors on which it is zero. There are no non-zero maps $\bI^r \to \bI^s$ with $s>r$, so we have $a(i) \le r$ for all $i$. Suppose $a(i)<r$ for all $i$. For any $m \in \bN$, the given injection induces an injection $(\bI^r)^{\fS^m} \to J^{\fS^m}$ of $\bK_m$-vector spaces. The domain has dimension $\binom{m}{r}$, while the target has dimension $\sum_{i=1}^n \binom{m}{a(i)}$. We thus have $\binom{m}{r} \le \sum_{i=1}^n \binom{m}{a(i)}$, which contradicts $a(i)<r$. We conclue that $a(i)=r$ for some $i$. Since $\End(\bI^r)$ consists of scalar multiples of the identity, the composition $\bI^r \to J \to \bI^{a(i)}=\bI^r$ is a non-zero multiple of the identity. The inverse map splits our injection.
\end{proof}

\begin{corollary}
\label{cor:indecomposable}
Every injective object of $\cA$ is a standard object. Moreover, every indecomposable injective is of the form $\bI^r$ for some $r$.
\end{corollary}

\begin{proof}
Since $\cA$ is locally noetherian, every injective is a direct sum of indecomposable injective objects. It is clear that $\bI^r$ is indecomposable, since $\End_{\cA}(\bI^r)$ is a field (see \S \ref{ss:groth}). Since every finitely generated object injects into a finite direct sum of $\bI^r$'s, it follows that these are the only indecomposable injectives.
\end{proof}

%

\begin{corollary}
Every standard object is injective in the category of smooth $\bk$-linear $\fS$-representation.
\end{corollary}

\begin{proof}
This is implied by the theorem and Proposition~\ref{prop:adjunction}.
\end{proof}

\begin{remark} \label{rmk:inj}
Let $\cC$ be the category of smooth $\bk$-linear $\fS$-representations. If $\bk$ has characteristic~0 then the $\bV^1=\bk^{\infty}$ is injective in $\cC$, and every indecomposable injective is a direct summand of a tensor power of $\bV^1$; this follows from \cite[\S 6.2.11]{infrank}, or from \cite[Proposition~2.2.10]{symc1} and the (easy) equivalence $\Mod_K=\cC$. In this case, the above corollary is not so interesting.

On the other hand, when $\bk$ has positive characteristic, the above corollary is the only explicit construction of injectives in $\cC$ that either of us know. We note that the injecitivity of $\bI^0$ implies that the $\FI$-module $[n] \mapsto \bk(\xi_1,\ldots,\xi_n)$ is injective; one can write down similar injective $\FI$-modules using $\bI^r$.
\end{remark}

\subsection{Standard resolutions over finite fields} \label{ss:finite}

The existence of standard resolutions (Corollary~\ref{cor:res}) is very important, but so far we have only established it for $\bk$ infinite. We now treat the case where $\bk$ is finite.

\begin{proposition}
Let $\bk'/\bk$ be an algebraic field extension and let $M$ and $N$ be objects of $\cA_{\bk}$, with $M$ finitely generated. We then have a natural isomorphism
\begin{displaymath}
\bk' \otimes_{\bk} \Ext^i_{\cA_{\bk}}(M, N) \to \Ext^i_{\cA_{\bk'}}(M \otimes_{\bk} \bk', N \otimes_{\bk} \bk')
\end{displaymath}
for all $i \ge 0$.
\end{proposition}

\begin{proof}
First notice that because $\bk'/\bk$ is algebraic, we have $\bK(\bk') = \bk' \otimes_{\bk} \bK$. In particular, $M \otimes_{\bk} \bk'$ is a vector space over $\bK(\bk')$, and is thus an object of $\cA_{\bk'}$.

Now, the usual adjunction gives
\begin{displaymath}
\Hom_{\cA_{\bk'}}(M \otimes_{\bk} \bk', N \otimes_{\bk} \bk') = \Hom_{\cA_{\bk}}(M, N \otimes_{\bk} \bk').
\end{displaymath}
Note that $N \otimes_{\bk} \bk'$ is simply isomorphic to a direct sum of copies of $N$ as an object of $\cA_{\bk}$. Since $M$ is finitely generated, it follows that we have
\begin{displaymath}
\Hom_{\cA_{\bk}}(M, N \otimes_{\bk} \bk') = \Hom_{\cA_{\bk}}(M, N) \otimes_{\bk} \bk'.
\end{displaymath}
This establishes the $i=0$ case of the proposition.

Choose an injective resolution $N \to I^{\bullet}$. Since $I^r$ is standard (Corollary~\ref{cor:indecomposable}), it follows that $\bk' \otimes_{\bk} I^r$ is standard, and thus injective (Theorem~\ref{thm:standards-are-injective}). Hence $\bk' \otimes_{\bk} I^{\bullet}$ is an injective resolution of $\bk' \otimes_{\bk} N$. By the previous paragraph, we have an isomorphism of complexes
\begin{displaymath}
\bk' \otimes_{\bk} \Hom_{\cA_{\bk}}(M, I^{\bullet}) = \Hom_{\cA_{\bk'}}(M \otimes_{\bk} \bk', I^{\bullet} \otimes_{\bk} \bk').
\end{displaymath}
The left side computes $\bk' \otimes_{\bk} \Ext^{\bullet}_{\cA_{\bk}}(M, N)$, while the right side computes $\Ext^{\bullet}_{\cA_{\bk'}}(M \otimes_{\bk} \bk', N \otimes_{\bk} \bk')$. The result follows.
\end{proof}

\begin{proposition} \label{prop:res2}
Let $M \in \cA$ be finitely generated. Then $M$ has finite injective dimension. Moreover, there is a resolution
\begin{displaymath}
0 \to M \to I^0 \to \cdots \to I^n \to 0
\end{displaymath}
where each $I^i$ is a finitely generated standard object.
\end{proposition}

\begin{proof}
Let $n$ be the injective dimension of $\ol{\bk} \otimes_{\bk} M$ in $\cA_{\ol{\bk}}$, which is finite (Corollary~\ref{cor:res} and Theorem~\ref{thm:standards-are-injective}). If $N$ is any object of $\cA$ then the previous proposition gives
\begin{displaymath}
\ol{\bk} \otimes_{\bk} \Ext^i_{\cA}(M, N) = \Ext^i_{\cA_{\ol{\bk}}}(\ol{\bk} \otimes_{\bk} M, \ol{\bk} \otimes_{\bk} N),
\end{displaymath}
which vanishes for $i>n$. We conclude that $M$ has injective dimension $\le n$.

To establish the existence of the resolution, we proceed by induction on the injective dimension of $M$. If $M$ is injective then it is standard (Corollary~\ref{cor:indecomposable}), and the result is clear. Now suppose $M$ has injective dimension $>0$. We can find an injection $f \colon M \to I^0$ where $I^0$ is a finitely generated standard object (Theorem~\ref{thm:embed}). Since $I^0$ is injective (Theorem~\ref{thm:standards-are-injective}), it follows that $\coker(f)$ has smaller injective dimension than $M$. Thus, by the inductive hypothesis, we have a resolution of the desired form for $\coker(f)$. Prepending $I^0$ to this resolution gives the sought resolution for $M$.
\end{proof}

\subsection{The Grothendieck group} \label{ss:groth}

Let $\bF_r=\bk(\xi_1, \ldots, \xi_r)^{\fS_r}$. We note that $\End_{\cA}(\bI^r)=\bF_r$ by Proposition~\ref{prop:stdmap}. Thus $\Hom_{\cA}(\bI^r, -)$ is naturally a functor to $\bF_r$ vector spaces.

\begin{lemma} \label{lem:stdhom}
We have $\dim_{\bF_r} \Hom_{\cA}(\bI^r, \bI^s) = \binom{r}{s}$ for all $r$ and $s$.
\end{lemma}

\begin{proof}
By Proposition~\ref{prop:stdmap}, we have $\Hom_{\cA}(\bI^r, \bI^s) = (\bI^s)^{\fS_r \times \fS^r}$. Now, one easily sees that
\begin{displaymath}
(\bI^s)^{\fS^r} = \bigoplus_{S \subset [r]} \bk(\xi_1, \ldots, \xi_r) e_S,
\end{displaymath}
where the sum is over subsets of size $s$. This is a semilinear representation of $\fS_r$ over $\bk(\xi_1, \ldots, \xi_r)$. Thus its $\bk(\xi_1,\ldots,\xi_r)$-dimension, which is $\binom{r}{s}$, coincides with the $\bF_r$-dimension of the $\fS_r$-invariants (Proposition~\ref{prop:galois}), and so the result follows.
\end{proof}

\begin{lemma}
Let $M$ be an object of $\cA$. Then $\Ext^i_{\cA}(\bI^r, M)$ is naturally a vector space over $\bF_r$. If $M$ is finitely generated then it is finite dimensional for all $i$ and vanishes for $i$ large.
\end{lemma}

\begin{proof}
Let $M \to I^{\bullet}$ be an injective resolution of $M$. Then $\Ext^{\bullet}_{\cA}(\bI^r, M)$ is computed by $\Hom_{\cA}(\bI^r, I^{\bullet})$, which is a complex of $\bF_r$-vector spaces. Thus each $\Ext$ group is a $\bF_r$ vector space. If $M$ is finitely generated then one can take $I^k$ to be a finitely generated standard object for all $k$ and vanish for $k \gg 0$ (Proposition~\ref{prop:res2}), which yields the result. (Note that $\Hom_{\cA}(\bI^r, I^k)$ is finite dimensional over $\bF_r$ by Lemma~\ref{lem:stdhom}.)
\end{proof}

\begin{theorem} \label{thm:groeth}
Let $\rK$ be the Grothendieck group of the category $\cA^{\fgen}$ of finitely generated objects of $\cA$. Then the classes $[\bI^r]$ form an integral basis of $\rK$.
\end{theorem}

\begin{proof}
The classes span $\rK$ by Proposition~\ref{prop:res2}. We must prove linear independence. Define
\begin{displaymath}
\lambda_r \colon \rK \to \bZ, \qquad
\lambda_r([M]) = \sum_{i \ge 0} (-1)^i \dim_{\bF_r} \Ext^i_{\cA}(\bI^r, M).
\end{displaymath}
This is well-defined by the lemma. We have $\lambda_r([\bI^s])=\binom{r}{s}$ by Lemma~\ref{lem:stdhom} and the injectivity of $\bI^s$. The result now follows. Indeed, suppose that $\sum_{k=r}^s a_k [\bI^k]=0$ is a linear relation with $a_r \ne 0$. Applying $\lambda_r$, we obtain $a_r=0$, a contradiction.
\end{proof}

\begin{remark}
It follows from the above discussion that $\rK$ admits a natural bilinear pairing $\langle, \rangle$ satisfying $\langle [\bI^r], [\bI^s] \rangle = \binom{r}{s}$. It is not clear to us how to describe $\langle [M], [N] \rangle$ directly in terms of $M$ and $N$: one would like to take an alternating sum of dimensions of Ext groups, but this does not exactly make sense, at least in a naive sense.
\end{remark}

\section{Semilinear representations and $\FIR^{\op}$-modules} \label{s:fir}

\subsection{Definitions, statements, and examples}

Let $\bW^n$ be the $\bK$-vector space with basis $e_{i_1, \ldots, i_n}$, where $i_1, \ldots, i_n$ are distinct elements of $[\infty]$, which is naturally a smooth $\bk$-linear representation of $\fS$. Let $\bJ^n=\bK \otimes_{\bk} \bW^n$, which is an object of $\cA$. It is easy to see that if $V$ is a representation of $\fS$ then $\Hom_{\fS}(\bW^n, V)=V^{\fS^n}$. Thus, by Proposition~\ref{prop:adjunction}, we have $\Hom_{\cA}(\bJ^n, M) = M^{\fS^n}$ for $M \in \cA$. In particular, we have
\begin{displaymath}
\Hom_{\cA}(\bJ^n, \bJ^m) \cong \begin{cases}
0 & \text{if $n<m$} \\
\bk(\xi_1, \ldots, \xi_n)^{\oplus n!/(n-m)!} & \text{if $n \ge m$}.
\end{cases}
\end{displaymath}
To understand the $n \ge m$ case, simply observe that $(\bJ^m)^{\fS^n}$ consists of linear combinations of the basis vectors $e_{i_1,\ldots,i_m}$ with $i_1,\ldots,i_m \in [n]$ with coefficients in $\bK_n$.  

Inspired by the calculation above, we define a $\bk$-linear category, denoted $\FIR$, as follows.  The objects are finite sets.  The space $\Hom_{\FIR}(S, T)$ of morphisms consists of all formal sums $\sum_{\phi \colon S \to T} c(\phi) [\phi]$ where $\phi$ varies over injections $S \to T$ and $c(\phi)$ belongs to the field $\bk(t_i)_{i \in T}$. Given $S \to T$ given by $c(\phi) [\phi]$ and $T \to U$ given by $c(\psi) [\psi]$, the composition is $\psi_*(c(\phi)) c(\psi) [\psi \circ \phi]$; here $c(\phi)$ is a rational functions in variables indexed by $T$, and $\psi_*(c(\phi))$ denotes the result of applying $\psi$ to the indices of the variables. We note that the full subcategory on objects $[n] = \{1,2, \ldots, n\}$ with $n \in \bN$ form a skeleton of the category $\FIR$. We will now identify $\FIR$ with this skeleton.

There is a $\bk$-linear functor $\cJ \colon \FIR^{\op} \to \cA$, described as follows. The object $[n]$ is taken to $\bJ^n$. A morphism $[n] \to [m]$ given by $\sum_{\phi \colon [n] \to [m]} c(\phi) [\phi]$ is taken to the unique morphism $\bJ^m \to \bJ^n$ that maps $e_{1,\ldots,m}$ to $\sum_{\phi \colon [n] \to [m]} c(\phi)(\xi_1, \ldots, \xi_m) e_{\phi(1), \ldots, \phi(n)}$. Here $c(\phi)(\xi_1, \ldots, \xi_m)$ is the result of substituting $\xi_i$ for $t_i$. 

An {\bf $\FIR^{\op}$-module} is a $\bk$-linear functor from $\FIR^{\op}$ to the category of $\bk$-vector spaces; a morphism of modules is a natural transformation. Let $\cB$ be the category of $\FIR^{\op}$-modules. We let $\bP^n$ denote the principal projective $\FIR^{\op}$-module given by $\bP^n([m]) = \Hom_{\FIR^{\op}}([n], [m])$; as the name suggests, these are projective objects of $\cB$. Via the functor $\cJ$, we obtain the following natural functors:
\begin{align*}
&\Psi \colon \cB^{\op} \to \cA, \qquad N \mapsto \Hom_{\cB}(N, \cJ) \\
&\Phi \colon \cA^{\op} \to \cB, \qquad M \mapsto \Hom_{\cA}(M, \cJ). 
\end{align*} We now state our main result.

\begin{theorem}
\label{thm:main}
We have the following:
\begin{enumerate}
\item The category $\cB$ is coherent (and not noetherian), and the abelian category $\cB^{\rm fp}$ of finitely presented $\FIR^{\op}$-modules is artinian.	
\item The functors $\Phi$ and $\Psi$ induce mutually quasi-inverse contravariant equivalences between $\cB^{\rm fp}$ and $\cA^{\rm fg}$. Here $\cA^{\rm fg}$ is the full subcategory of $\cA$ spanned by finitely generated objects.
\end{enumerate}
\end{theorem}

Before embarking on its proof, we look at some examples.

\begin{example} \label{ex:p1}
Let us look at the structure of $\bP^1$, the first principal projective. We have $\bP^1([m])=\Hom_{\FIR}([m], [1])$. If $m>1$ then there are no injections $[m] \to [1]$, and this space vanishes. For $m \le 1$, there is a unique injection $[m] \to [1]$, and so $\bP^1([m])=\bk(t)$. If $V$ is any $\bk$-linear subspace of $\bk(t)$ then we can define a submodule $N_V$ of $\bP^1$ by $N_V([0])=V$ and $N_V([m])=0$ for $m>0$. Every proper subobject of $\bP^1$ is one of the $N_V$'s. Indeed, if $N$ is a subobject of $\bP^1$ that is non-zero in degree~1 then $N([1])$ is an $\End_{\FIR}([1])=\bk(t)$-submodule of $\bP^1([1])=\bk(t)$, and thus all of $\bP^1([1])$. Since $\bP^1$ is generated in degree~1, we get $N=\bP^1$.

This analysis shows that $\bP^1$ is not a noetherian object: if $V_1 \subset V_2 \subset \cdots$ is an increasing chain of $\bk$-subspaces of $\bk(t)$ then $N_{V_1} \subset N_{V_2} \subset \cdots$ is an increasing chain of subobjects of $\bP^1$. It turns out that only the $N_V$ with $V$ finite dimensional can be realized as the kernel of a map from $\bP^1$ to a finite sum of principal projectives. This is essentially the content of the coherence statement in Theorem~\ref{thm:main} in this case.
\end{example}

\begin{example} \label{ex:I1}
The theorem, and the previous example, implies that there is an order reversing isomorphism
\begin{displaymath}
\Phi \colon \{\text{finite dimensional $\bk$-subspaces of $\bk(t)$} \} \to \{\text{non-zero submodules of $\bI^1$}  \}
\end{displaymath}
of posets satisfying $\dim_{\bk} V = \codim_{\bK} \Phi(V)$ (in other words, $\Phi$ is a Galois correspondence). More explicitly, let $V$ be a finite dimensional $\bk$-subspace of $\bk(t)$. Let $a_1(t), \ldots, a_n(t)$ be a basis of $V$. Then $\Phi(V)$ is the submodule of $\bI^1$ generated by the determinant of the following $(n+1) \times (n+1)$ matrix:
\begin{displaymath}
\begin{bmatrix}
e_1 & e_2 & \cdots & e_{n+1} \\
a_1(\xi_1) & a_2(\xi_2) & \cdots & a_1(\xi_{n+1}) \\
\vdots & \vdots & \ddots & \vdots \\
a_n(\xi_1) & a_n(\xi_2) & \cdots & a_n(\xi_{n+1})
\end{bmatrix}.
\end{displaymath}
The theorem above ensures that all submodules of $\bI^1$ are of this form.
\end{example}

\subsection{The proof}

We now start on the proof of Theorem~\ref{thm:main}. We require a few lemmas.

\begin{lemma}
\label{lem:jstruc}
We have an isomorphism $\bJ^n \cong (\bI^n)^{\oplus n!}$.
\end{lemma}

\begin{proof}
 Let $\bK_n = \bk(\xi_1, \ldots, \xi_n)$ and let $\bF_n=\bK_n^{\fS_n}$. By Proposition~\ref{prop:galois}, if $V$ is a $\bK_n$-semilinear representation of $\fS_n$ then the natural map $\bK_n \otimes_{\bF_n} V^{\fS_n} \to V$ is an isomorphism. Thus if $M \in \cA$ then we have a canonical isomorphism
\begin{displaymath}
\Hom_{\cA}(\bJ^n, M) = \Hom_{\cA}(\bI^n, M) \otimes_{\bF^n} \bK_n.
\end{displaymath}
Indeed, note that $\Hom_{\cA}(\bJ^n, M)=M^{\fS^n}$ is a $\bK_n$-semilinear representation of $\fS_n$, and its $\fS_n$-invariant space is $\Hom_{\cA}(\bI^n, M)$. We thus see that, as functors to $\bk$-vector spaces, we have an isomorphism
\begin{displaymath}
\Hom_{\cA}(\bJ^n, -) \cong \Hom_{\cA}(\bI^n, -)^{\oplus n!},
\end{displaymath}
and so $\bJ^n \cong (\bI^n)^{\oplus n!}$ by Yoneda's lemma.
\end{proof}

\begin{lemma}
\label{lem:phi-is-exact}
The functor $\Phi$ is exact.
\end{lemma}

\begin{proof}
This follows immediately from the fact that $\bJ^n$ is injective, which is a consequence of Lemma~\ref{lem:jstruc} and Theorem~\ref{thm:standards-are-injective}.
\end{proof}

\begin{lemma}
\label{lem:projective-equivalence}
The contravariant functor $\Phi$ takes $\bJ^n$ to $\bP^n$, and induces an isomorphism
\begin{displaymath}
\Hom_{\cA}(\bJ^n, \bJ^m) \to \Hom_{\cB}(\bP^m, \bP^n).
\end{displaymath}
Moreover, $\Phi$ is fully faithful on standard objects, takes standard objects to projective objects, and all finitely generated  projective objects in $\cB$ can be obtained in this way.
\end{lemma}

\begin{proof}
We have $\Phi(\bJ^n)([m]) = \Hom_{\cA}(\bJ^n, \bJ^m) = \Hom_{\FIR^{\op}}([n], [m])$. Here the second equality follows from the definition of $\FIR$. These equalities are functorial in $[m]$ which implies that we have $\Phi(\bJ^n) = \bP^n$. For the second assertion, note that we  have
\begin{displaymath}
\Hom_{\cA}(\bJ^n, \bJ^m) = \Hom_{\FIR^{\op}}([n], [m]) = \bP^n([m]) = \Hom_{\cB}(\bP^m, \bP^n),
\end{displaymath}
where the first equality follows from the definition of $\FIR$, and the last equality is implied by Yoneda lemma. This proves that
\begin{displaymath}
\Hom_{\cA}(\bJ^n, \bJ^m) \to \Hom_{\cB}(\bP^m, \bP^n)
\end{displaymath}
is an isomorphism.
	
Since $\Phi$ is additive, we see from Lemma~\ref{lem:jstruc} that $\Phi(\bI^n)$ is a direct summand of $\Phi(\bJ^n)=\bP^n$.  Hence $\Phi(\bI^n)$ is a projective object. The map
\begin{displaymath}
\Hom_{\cA}(\bI^n, \bI^m) \to \Hom_{\cB}(\Phi(\bI^m), \Phi(\bI^n))
\end{displaymath}
is a direct summand of the map in the previous paragraph, and so is an isomorphism. This shows that $\Phi$ is fully faithful on standard objects. 
	
Finally, note that any finitely generated projective object $P$ is a direct summand of a projective of the form $Q=\bigoplus_{i=1}^k \bP^{n_i}$. In particular, $P$ can be written as a cokernel of a split surjection $Q \to Q$. By the first paragraph, we conclude that $P$ is of the form $\Phi(I)$ for some direct summand $I$ of $\bigoplus_{i=1}^k \bJ^{n_i}$. Since $\bJ^n$ is injective, so is $I$ and thus $I$ is standard (Corollary~\ref{cor:indecomposable}).
\end{proof}

\begin{proof}[Proof of Theorem~\ref{thm:main}]
We first show that $\cB$ is coherent. Suppose that $f \colon P \to Q$ is a map of finitely generated projective $\FIR^{\op}$-modules; we must show that $\ker(f)$ is finitely generated. We can write $P=\Phi(P')$ and $Q=\Phi(Q')$ with $P'$ and $Q'$ finitely generated standard objects. By the previous lemma, $\Phi$ is fully faithful on standard modules, so we have $f=\Phi(f')$ for some $f' \colon Q' \to P'$. Let $K'$ be the cokernel of $f'$, so that $\Phi(K')$ is the kernel of $f$ (Lemma~\ref{lem:phi-is-exact}); we must show that this is finitely generated. Since $K'$ is a finitely generated object of $\cA$, we can find an injection $K' \to I$ for some finitely generated standard object $I$ (Theorem~\ref{thm:embed}). By Lemma~\ref{lem:phi-is-exact}, this yields a surjection $\Phi(I) \to \Phi(K')$, and since we know $\Phi(I)$ is finitely generated, it follows that $\ker(f)$ is finitely generated. This proves that $\cB$ is coherent.
	
We now prove that the functor
\begin{displaymath}
\Phi \colon (\cA^{\rm fg})^{\op} \to \cB^{\rm fp}
\end{displaymath}
given by $\Hom_{\cA}(-, \cJ)$ is an equivalence of categories. By Lemma~\ref{lem:phi-is-exact}, $\Phi$ is exact. Moreover, by Lemma~\ref{lem:projective-equivalence}, $\Phi$ induces an equivalence on the categories of projective objects. By Theorem~\ref{thm:embed}, there are enough projective objects in the source category, and it is clear that there are enough projective objects in the target category. We conclude that $\Phi$ is an equivalence. The theorem now follows.
\end{proof}

\subsection{An application} \label{ss:app}

Let $M \in \cA$ be a finite generated. By Proposition~\ref{prop:res2}, we have a resolution
\begin{displaymath}
0 \to M \to I^0 \to \cdots \to I^n \to 0
\end{displaymath}
where each $I^i$ is a finitely generated standard object. If $\bk$ is infinite, Corollary~\ref{cor:res} shows that we can take $I^i$ to be generated in degrees $\le g-i$ where $g$ is the generation degree of $M$ (and thus $I^i=0$ for $i>g$). We now show that this more precise statement continues to hold if $\bk$ is finite. We do this not because we care so much about the end result, but simply to illustrate how the $\FIR$ side is much easier than the $\fS$-side, and so Theorem~\ref{thm:main} really lets us understand semilinear representations better than before.

\begin{proposition}
Let $M$ be a finitely generated $\FIR^{\op}$-module, and let $n$ be maximal such that $M([n]) \ne 0$. Then there is a surjection $f \colon P \to M$ where $P$ is a finitely generated projective $\FIR^{\op}$-module such that $\ker(f)$ is supported in degrees $<n$.
\end{proposition}

\begin{proof}
Let $K=\bk(t_1, \ldots, t_n)$. We have $\End_{\cB}(\bP^n)=\End_{\FIR}([n])=K[\fS_n]$. Let $V=M([n])$, which is a $K$-semilinear representation of $\fS_n$, and let $Q=\bP^n \otimes_{K[\fS_n]} V$. Since the category of $K$-semilinear representations of $\fS_n$ is semisimple, it follows that $Q$ is projective. There is a canonical map $Q \to M$ of $\FIR^{\op}$-modules which is an isomorphism in degree $n$. The cokernel is supported in smaller degree, and thus a quotient of a finitely generated projective $Q'$ supported in degrees $<n$. It follows that we have a surjection $f \colon P \to M$ with $P=Q \oplus Q'$ with the required properties.
\end{proof}

\begin{corollary}
Let $M$ be a finitely presented $\FIR^{\op}$-module supported in degrees $\le n$. Then we have a resolution
\begin{displaymath}
0 \to P_n \to \cdots \to P_0 \to M \to 0
\end{displaymath}
where $P_i$ is a finitely generated projective supported in degrees $\le n-i$.
\end{corollary}

\begin{proof}
Simply apply the proposition iteratively.
\end{proof}

\begin{corollary} \label{cor:res2}
Let $M$ be a finitely generated object of $\cA$ generated in degrees $\le n$. Then we have a resolution
\begin{displaymath}
0 \to M \to I^0 \to \cdots \to I^n \to 0
\end{displaymath}
where $I^i$ is a finitely generated standard object generated in degrees $\le n-i$.
\end{corollary}

\begin{proof}
Move the previous corollary across the equivalence of categories.
\end{proof}

\end{document}